\documentclass[12pt,reqno]{amsart}
\usepackage{amsmath, amsthm, amscd, amsfonts, amssymb, graphicx, xcolor, setspace}
\usepackage[pagebackref,colorlinks=true,linkcolor=blue,citecolor=blue,urlcolor=red, hypertexnames=true]{hyperref}
\usepackage[abbrev,backrefs,nobysame]{amsrefs}
\usepackage{color}

\long\def\blue#1{\textcolor {blue}{#1}}

\usepackage[tmargin=1.15in,bmargin=1.15in,rmargin=1.1in,lmargin=1.1in]{geometry}

\newtheorem{theorem}{Theorem}[section]
\newtheorem{lemma}[theorem]{Lemma}
\newtheorem{proposition}[theorem]{Proposition}
\newtheorem{corollary}[theorem]{Corollary}
\theoremstyle{definition}
\newtheorem{definition}[theorem]{Definition}

\newtheorem{problem}[theorem]{Problem}
\theoremstyle{remark}
\newtheorem{remark}[theorem]{Remark}

\numberwithin{equation}{section}

\newcommand{\D}{\ensuremath{\mathbb D}}
\newcommand{\K}{\ensuremath{\mathbb K}}
\newcommand{\T}{\ensuremath{\mathbb T}}
\newcommand{\R}{\ensuremath{\mathbb R}}
\newcommand{\Z}{\ensuremath{\mathbb Z}}
\newcommand{\C}{\ensuremath{\mathbb C}}

\newcommand{\N}{\ensuremath{\mathbb N}}

\newcommand{\pss}[2]{\ensuremath{{\langle #1,#2\rangle}}}
  \newcommand{\norme}[1]{\left\| #1 \right\|}
  \newcommand{\abs}[1]{\left\vert #1 \right\vert}
  \newcommand{\ps}[2]{\left\langle #1,#2 \right\rangle}

\newcommand{\eps}{\varepsilon}
\newcommand{\ld}{\lambda}

  \newcommand{\B}[1]{\mathcal{B}(#1)}

\newcommand{\Jam}{Jamison}

\newcommand{\nk}{\ensuremath{(n_k)_{k\ge 0}\,}}

\renewcommand{\Re}{\operatorname{Re}}

\newcommand{\Arg}{\operatorname{Arg}}

\begin{document}
\linespread{1.1}\selectfont{}

\title[Escaping a neighborhood]{Escaping a neighborhood along a prescribed sequence in Lie groups and Banach algebras}

\author[Badea]{Catalin Badea}
\address[C.~Badea]{Univ. Lille, CNRS, UMR 8524 - Laboratoire Paul Painlev\'{e}, France}
\email{catalin.badea@univ-lille.fr}
\urladdr{http://math.univ-lille1.fr/~badea/}

\author[Devinck]{Vincent Devinck}
\address[V.~Devinck]{Univ. d'Artois, Laboratoire de Math\'{e}matiques de Lens, FR CNRS 2956, France}
\email{devinck@math.cnrs.fr}

\author[Grivaux]{Sophie Grivaux}
\address[S.~Grivaux]{CNRS, Univ. Lille, UMR 8524 - Laboratoire Paul Painlev\'{e}, France}
\email{sophie.grivaux@univ-lille.fr}
\urladdr{http://math.univ-lille1.fr/~grivaux/}
\keywords{Jamison sequences, point spectrum, iterates, accretive operators, Banach algebras, Lie groups, groups with no small subgroups, minimal metric}
 \subjclass[2010]{47A10, 47A12, 47A60, 22E15}
 \thanks{This work was supported in part by
 the project FRONT of the French
National Research Agency (grant ANR-17-CE40-0021) and by the Labex CEMPI (ANR-11-LABX-0007-01).}

\onehalfspace

\begin{abstract}
It is shown that Jamison sequences, introduced in 
$2007$ by Badea and Grivaux,
arise naturally in the study of topological groups with no small subgroups, of Banach or normed algebra elements whose powers are close to identity along subsequences, and in characterizations of (self-adjoint) positive operators by the accretiveness of some of their powers. The common core of these results is a description of those sequences for which non-identity elements in Lie groups or normed algebras escape an arbitrary small neighborhood of the identity in a number of steps belonging to the given sequence. Several spectral characterizations of Jamison sequences are given and other related results are proved.
\end{abstract}

\maketitle

\section{Introduction}\label{Sec0}
\subsection{Jamison sequences} 
The main characters of this manuscript are the so-called \emph{Jamison sequences} of integers. This notion was introduced in the paper \cite{BaGr}, following the original work of Jamison in \cite{Jam} and the subsequent extensions of Ransford in \cite{Ran} and Ransford and Roginskaya in \cite{RR}. It is part of the general study of the relationships between the geometry of a (complex, separable) Banach space $X$, the growth of the iterates $T^{n}$ of a bounded operator $T\in\mathcal{B}(X)$, and the size of (parts of) its spectrum. More precisely, the following definition was introduced in \cite{BaGr}.

\begin{definition}[Jamison sequences]\label{def:jam}
A sequence of integers $(n_k)_{k\ge 0}$  with $n_0 = 1$ is said to be a \emph{Jamison sequence} if the following spectral property holds: for any bounded
linear operator $T$ on a complex separable Banach space
such that $\sup_{k\geq 0}\|T^{n_{k}}\|<+\infty $, the set of eigenvalues of modulus $1$ of $T$ is countable.
\end{definition}

Jamison \cite{Jam} proved the result that the set of eigenvalues of modulus $1$ of a power-bounded operator on a complex separable Banach space is countable. This can be formulated as ``$n_k=k+1$ is a Jamison sequence'', whence the terminology used in Definition~\ref{def:jam}. We also know (see for example \cite{BaGrsmf}) that the sequence $n_k = 2^k$ is a Jamison sequence. More generally, sequences with \emph{bounded quotients}, that is sequences $(n_k)$ with $n_0 = 1$ and 
$\sup_{k\ge 0} \frac{n_{k+1}}{n_k} \le c < +\infty$ are Jamison sequences. We refer to Subsection \ref{Sec2} below for more information and for other examples of Jamison sequences. See also \cite{BaGrsmf} and the references therein for a recent survey of results concerning Jamison sequences and related notions, as well as for many examples of Jamison sequences. On the other hand, the sequences given by $n_k = k!$ or by $n_0=1$ and $n_k = 2^{2^k}$ for $k\ge 1$ are not Jamison sequences. 

The following characterization has been obtained by two authors of the present paper in \cite{BaGr2}. 

\begin{theorem}[characterizing \Jam\ sequences; \cite{BaGr2}]
\label{thm:1.2}
A sequence of integers $(n_k)_{k\ge 0}$  with $n_0 = 1$ is a Jamison sequence if and only if there exists $\eps > 0$ such that  for every $\lambda \in\T\setminus\{1\}$,
\begin{equation}\label{eq:sup}
\sup_{k\ge 0}|\lambda ^{n_{k}}-1|\ge \varepsilon.
\end{equation}
\end{theorem}

It will be helpful in this paper to pass from the (qualitative) definition of Jamison sequences to the (quantitative) definition of Jamison pairs. 

\begin{definition}[Jamison pairs and Jamison constants]
\label{def:jam1}
Given a sequence $(n_k)_{k\ge 0}$ of integers with $n_0 = 1$ and a real number $\eps > 0$, we say that $(\nk,\eps)$ is a \emph{Jamison pair} whenever \eqref{eq:sup} holds
for every $\lambda \in\T\setminus \{1\}$.
If \nk is a Jamison sequence, the supremum of all $\varepsilon$'s such that $(\nk,\eps)$ is a Jamison pair is called the \emph{Jamison constant} of $\nk$.
\end{definition} 

The above condition \eqref{eq:sup} can be interpreted as a Diophantine approximation condition. Indeed, using 
the fact that the quantities
$ \left|e^{i2\pi n\alpha}-1\right|$ and  $\langle\langle n\alpha\rangle\rangle$, $\alpha \in \R$, are comparable, 
where $\langle\langle \cdot\rangle\rangle$ denotes the distance to the closest integer, one can interpret \eqref{eq:sup} for $\ld = e^{i2\pi\alpha}$ as the impossibility to well approximate $\alpha$ by rationals with prescribed denominators from the sequence \nk. An analogue of this Diophantine approximation condition also has been used in \cite{Dev} by one of the authors to give a version of Theorem~\ref{thm:1.2} for $C_0$-semigroups. See also Section~\ref{sect:inR} for more on Jamison sequences in $\R$. 

The same condition  \eqref{eq:sup} has the following dynamical interpretation in terms of nontrivial circle rotations $R_\lambda\,:\,\T\ni z\mapsto \lambda z \in \T$. If $\nk$ is a Jamison sequence, then there exists a neighborhood $V$ of $1$ in $\T$ such that $R_\lambda^{n_k}(1)$ ``escapes $V$'', that is $R_\lambda^{n_k}(1) \notin V$ for some $k\ge 0$.

\subsection{What this paper is about}
The aim of the present paper is to obtain some surprisingly general characterizations of Jamison sequences and Jamison pairs. We summarize the main results obtained in this paper as follows.
\begin{enumerate}
\item A sequence $(n_k)$ is Jamison if and only if, for any Lie group $G$, there is a neighborhood $U$ of the identity element $e$ in $G$ such that for any $g\neq e$ in $G$, the sequence $(g^{n_k})$ ``escapes $U$'', that is $g^{n_k} \notin U$ for some $k\ge 0$. The same holds true for any linear Banach-Lie group. 

\item If \nk has bounded quotients, then it satisfies the same property as in (1) for any Banach-Lie group $G$, and for any topological group admitting a \emph{minimal metric} (a notion introduced recently by Rosendal \cite{Rosendal}).

\item A pair $((n_k),\eps)$ is Jamison with $\eps\le1$ if and only if for any complex normed algebra $A$ with identity $e$ and with any $a\neq e$ in $A$, the sequence $(a^{n_k})$ escapes the ball $B(e,r)$ for every $r<\eps$.

\item The pair $((n_k),\sqrt{2})$ is Jamison if and only if any Hilbert space operator $T$ such that 
$\Re T^{n_k} > 0$ for all $k\ge 0$ is a positive invertible operator. The pair $((n_k),\sqrt{2})$ is a ``strict'' Jamison pair (i.e.\, strict inequality holds in \eqref{eq:sup}), if and only if any Hilbert space operator $T$ such that 
$\Re T^{n_k} \ge 0$ for all $k\ge 0$ is a positive operator. 
\end{enumerate}

As explained in Section~\ref{Sec1} below, these results are generalizations of classical results obtained by Got\^o and Yamabe (for Lie groups), Chernoff (for normed algebras) and Shiu (for positive operators) in the case $n_k := 2^k$.

\subsection{Organization of the paper}
We end this introduction by presenting some examples of Jamison sequences, including a sharp estimate of the Jamison constant of a sequence with bounded quotients. In the next section we state the main results of this paper and the background behind them. Section~\ref{Sec3} is devoted to the study of sequences verifying an analogue of the ``escape property'' in normed algebras: we prove there Theorem~\ref{thm:BanAlg} and state some consequences. 
The characterization of sequences verifying an ``escape property'' for Lie groups (Theorem~\ref{Thm:Lie1}) is proved in Section~\ref{Sec4}. Section~\ref{Sec5} is devoted to the proof of Theorem~\ref{thm:Lie2}, while the results of Subsection~\ref{subsec:1.5} below about sequences characterizing positive operators are proved in the final Section~\ref{Sec6}.

\subsection{Jamison sequences: some examples}\label{Sec2}
 So as not to interrupt the flow of the presentation, we have postponed to this subsection our discussion of some examples of Jamison sequences. The reader impatient to know the statement of the main results of this paper can go directly to Section~\ref{Sec1}.

The first class of examples of Jamison sequences that we consider here is the class of sequences with bounded quotients. The fact that such sequences are Jamison was proved in \cite{RR}*{Theorem\,1.5}. Using Theorem~\ref{thm:1.2} we give here a quick proof of this result, including an estimate of the Jamison constant.

\begin{proposition}[Sequences with bounded quotients are Jamison]\label{prop:bddq}
Let \nk be a sequence of integers with $n_0 = 1$ and 
$$\sup_{k\ge 0} \frac{n_{k+1}}{n_k} \le c < +\infty .$$ Then 
$(\nk, 2\sin(\pi/(c+1)))$ is a Jamison pair.
\end{proposition}

The constant $2\sin(\pi/(c+1))$ from Proposition \ref{prop:bddq} is sharp, as the example of the sequence given by $n_k = c^k$, for a positive integer $c$, shows. Indeed, for $\mu = e^{2i\pi/(c+1)}$ we have
$$\sup_{k\geq 0}\left|\mu^{n_k}-1\right| = \sup_{k\geq 0}\left|e^{2ic^k\pi/(c+1)}-1\right| = 2\sin\left(\frac{\pi}{c+1}\right).$$
In particular, for the sequence $n_k=2^k$ of powers of $2$, the Jamison constant is exactly $\sqrt{3}=2\sin(\pi/3)$. The same value $\sqrt{3}$ is the Jamison constant of the sequence $n_k=k+1$ (see \cite{BaGrsmf}*{Exemple\,2.11}). This is one possible explanation for the fact that several results described below, known for the sequence given by $n_k = k+1$, are valid also for the sequence of powers of $2$.  

 \begin{proof}[Proof of Proposition \ref{prop:bddq}]
 Suppose that $n_0 = 1$ and $n_{k+1} \le cn_k$  if $k\ge 0$. Suppose that $\ld = e^{i\theta}$, with $\abs{\theta} \le \pi$, satisfies $\abs{\ld^{n_k}-1} \le \eps < 2\sin(\pi/(c+1))$ for every $k\ge 0$. For the sake of contradiction, suppose that $\ld \neq 1$. Without loss of generality we can assume that $\theta \in ]0,\pi]$. Since $n_0=1$, we have $2\sin(\theta/2) = \abs{e^{i\theta}-1} < 2\sin(\pi/(c+1))$. Thus $0 < \theta < 2\pi/(c+1)$. Let $j$ be the smallest positive integer such that $n_{j+1}\theta \ge 2\pi/(c+1)$. Then 
 $$\frac{2\pi}{c+1} \le n_{j+1}\theta \le cn_j\theta \le 2\pi - \frac{2\pi}{c+1}$$
 and thus $\abs{\ld^{n_{j+1}}-1} = 2\sin\left(n_{j+1}\theta/2\right)\ge 2\sin(\pi/(c+1))$, a contradiction. Therefore $\theta=0$. It follows from Theorem~\ref{thm:1.2} that $(\nk, 2\sin(\pi/(c+1)))$ is a Jamison pair.
 \end{proof}

The next class of Jamison sequences, considered in \cite{BaGr2}, shows that not only the growth of the sequence matters, but also its arithmetical properties. 
Recall that  a set $\Sigma = \{\sigma_k : k\ge 0\}$ of real numbers is said to be
 \emph{dense modulo} $1$ if the set
$\Sigma + \Z = \{\sigma_k + n : k \ge 0, \, n\in \Z\}$
is dense in $\R$. For any $\eta > 0$, the set $\Sigma$ is said
to be $\eta$-\emph{dense modulo} $1$ if the set
$\Sigma+\Z$ intersects every open sub-interval of $\R$ of length greater than $ \eta$.

\begin{proposition}[arithmetic Jamison sequences; \cite{BaGr2}]\label{prop2}
Let $(n_k)_{k\ge 0}$ be a  sequence of integers with $n_0=1$. If
there exists a number $0<\eta < 1$
such that the set
$$D_{\eta} = \{\theta\in \R : (n_k\theta)_{k\ge 0} \text{ is not } \eta\text{-dense modulo } 1\}$$
is countable, then \nk is a Jamison sequence. In particular, if $(n_k\theta)_{k\ge 0}$ is dense modulo $1$ for every irrational $\theta$, then \nk is a Jamison sequence. 
\end{proposition}

It is more difficult to give here an estimate of the Jamison constant of $(n_k)_{k\ge 0}$ since the proof of Proposition \ref{prop2} is based upon a qualitative statement (\cite{BaGr2}*{Corollary\,2.11}): 

\begin{theorem}[a second characterization of \Jam\ sequences; \cite{BaGr2}]
\label{thm:1.6}
A sequence $(n_k)_{k\ge 0}$ of integers with $n_0 = 1$ is a Jamison sequence if and only if there exists an $\varepsilon > 0$ such that the set
$$ \Lambda_{\eps} := \{\ld \in \T : \sup_{k\ge0}\left|\lambda^{n_k} - 1\right| < \eps\}$$
is  countable. 
\end{theorem}
Several other examples of Jamison sequences are presented in 
\cites{RR,BaGr,BaGr2,BaGrsmf}.

\section{Background and main results}\label{Sec1}
We start by describing 
the sequences $(n_{k})_{k\ge 0}$ of positive integers which have the property that given a Lie group $G$ or a normed algebra $A$, powers $x^{n_{k}}$ of an element $x$ different from the identity element $e$ escape some prescribed neighborhood of $e$.

\subsection{NSS sequences in Lie groups}
It is well known that the sequence $n_{k}=k+1$, $k\ge 0$, has this escape property for every Lie group $G$. Indeed, real or complex Lie groups have no small subgroups, that is there exists a neighborhood of the identity element which contains no subgroup other than the trivial one. See for instance \cite{HofmannMorrisEMS}*{Proposition\,2.17} for the classical argument, or the proof of Theorem \ref{thm:Lie2} below. The standard terminology  is that Lie groups are \emph{NSS (No Small Subgroups)}.

As part of the solution, due to Gleason, Montgomery, Yamabe and Zippin, of Hilbert's fifth problem (the topological characterization of locally compact Lie groups), we know that, conversely, locally compact groups with no small subgroups are isomorphic to Lie groups. We refer the reader to the expositions in \cite{MontgomeryZippin} or \cite{Tao5} for this result and related aspects concerning Hilbert's fifth problem.   

It was proved in 1951 by Got\^o and Yamabe \cite{GotoYamabe} that a locally compact group $G$ with no small subgroups (so isomorphic to a Lie group) has the following property: for every $x\neq e$ in a sufficiently small neighborhood $U$ of the identity $e$ in $G$, there exists an integer $k$ such that $x^{2^k} \not\in U$. Note that powers of $2$ plays a special role in the construction of (weak) Gleason metrics (see for instance \cite{Tao5}), a fundamental toolkit in the solution of Hilbert's fifth problem. It is a natural question to ask which sequences of integers can replace the powers of $2$ in the result of Got\^o and Yamabe. 
By \emph{sequence of integers}, we will always mean a strictly increasing sequence $\nk$ of positive integers
with $n_0=1$.
\par\smallskip
In light of the preceding discussion we introduce the following definition:

\begin{definition}[NSS sequences for topological groups]
Let $G$ be a topological group and let \nk  be a sequence of integers with $n_0=1$. We say that $(G, \nk)$ is NSS if there exists a neighborhood $U$ of the identity $e$ of $G$ such that if $g^{n_k} \in U$ for every $k$, then $g=e$. We say that \nk is NSS for a class $\mathcal{C}$ of topological groups if $(G, \nk)$ is NSS for each group $G$ in the class $\mathcal{C}$.
\end{definition}

Our first main result is a characterization of NSS sequences for (real or complex) Lie groups, in surprisingly simple terms. We write $\T=\{\lambda \in\C\,;\,|\lambda |=1\}$.

\begin{theorem}[A characterization of NSS sequences, Lie groups version]\label{Thm:Lie1}
Let \nk be a sequence of integers. The following assertions are equivalent:
 \begin{enumerate}
  \item [(i)] $(n_{k})_{k\ge 0}$ is \emph{NSS} for the class of all Lie groups;
  \item[(ii)] $(\T,(n_{k})_{k\ge 0})$ is \emph{NSS};
\item[(iii)] \nk is a Jamison sequence.
 \end{enumerate}
 \end{theorem}
Notice that the equivalence of (ii) and (iii) of Theorem~\ref{Thm:Lie1} follows from Theorem~\ref{thm:1.2}.

\smallskip
\subsection{NSS sequences in normed algebras}
In order to prove Theorem \ref{Thm:Lie1}, an important step will be to prove that Jamison sequences are NSS for \emph{linear} Lie groups, i.e. for matrix groups. 
The groups $\textrm{GL}_{n}(\R)$ and $\textrm{GL}_{n}(\C)$ have No Small Subgroups; see for instance \cite{HofmannMorrisEMS}*{Proposition\,2.17} for an argument valid for all Lie groups. For complex matrices the following result has been proved as early as $1966$ by Cox \cite{Cox} (see also \cite{BeidCox}): if $M$ is a square matrix with complex entries such that $\sup_{n\ge 1} \|M^n - I\| < 1$, then $M$ is the identity matrix $I$. This result gives the maximal radius of the ball around $I$ in which $\textrm{GL}_n(\C)$ has no nontrivial subgroups. To show maximality, take $M=\delta I$ for small $\delta$. Cox's result has been extended to bounded linear operators on a Hilbert space by Nakamura and Yoshida \cite{NakamuraYoshida} and to arbitrary normed algebras by Hirschfeld \cite{Hirschfeld} and Wallen \cite{Wallen}. See also the related paper by Wils \cite{Wils}.
The fact that the group of invertible elements 
of a complex Banach algebra $A$ has no small subgroups was proved as early as $1941$ by Gelfand \cite{Gel}; see for instance the explanations in Kaplansky's book \cite{Kaplansky}*{p.\,88}. Notice that, by taking completions, one can always assume that the considered normed agebra is in fact a Banach algebra.

It seems that the first who considered, in the framework of normed algebras, powers of elements along subsequences was Chernoff \cite{Chernoff}. In $1969$, he proved that if $A$ is a complex normed algebra with unit $e$ and $a \in A$ is such that $\sup_{k\ge 0} \|a^{2^k} - e\| < 1$, then $a=e$. This is the normed algebras counterpart of the result of Got\^o and Yamabe \cite{GotoYamabe} for the sequence of powers of $2$. Chernoff's result has been extended/generalized in \cite{Gor} and \cite{KMOT}.

It is thus natural to introduce the following definition.  Recall that a unital normed algebra with unit $e$ satisfies $\norme{e} = 1$ and $\norme{xy} \le \norme{x}\norme{y}$ for every $x,y\in A$. 

\begin{definition}(NSS for normed algebras)\label{def:nss}
A triplet $(A, (n_k)_{k\ge 0}, \varepsilon)$ consisting of a normed unital algebra $A$ with unit $e$, a sequence of integers  $(n_k)_{k\ge 0}$ with $n_0 = 1$ and a positive real number $\varepsilon$ is said to be \emph{NSS} if the following implication holds true: the only element $a\in A$ satisfying $\sup_{k\ge 0} \|a^{n_k} - e\| < \varepsilon$ is $a=e$. 
\end{definition}

Here NSS stands again for \emph{No Small Subgroups}. Then Chernoff's result can be expressed shortly as ``$(A, (2^k)_{k\ge0},1)$ is NSS'' for any complex normed algebra $A$.
\par\smallskip
Our second main result, which extends \cite{Gor}*{Theorem\,2} and \cite{KMOT}*{Corollary\,4.2}, is the following:

\begin{theorem}[A characterization of NSS sequences, normed algebras version]\label{thm:BanAlg}
Let $\varepsilon > 0$ be a real number and let $(n_k)_{k\ge 0}$ be a sequence of integers with $n_0 = 1$. Then the following assertions are equivalent: 
\begin{itemize}
\item[(i)] $(A, (n_k)_{k\ge 0}, \varepsilon)$ is NSS for any complex normed algebra $A$;  

\item[(ii)] $(\C, (n_k)_{k\ge 0}, \varepsilon)$ is NSS; 

\item[(iii)] $(\nk,\eps)$ is a Jamison pair with $\eps \le 1$. 
\end{itemize}
\end{theorem}
Note that the assumption that $(\C, (n_k)_{k\ge 0}, \varepsilon)$ is NSS for the Banach algebra $\C$ of all complex numbers is a minimal requirement to have that $(A, (n_k)_{k\ge 0}, \varepsilon)$ is NSS for any complex normed algebra $A$. By considering $a=0$ in the definition of a Jamison pair, it is clear that the condition $\eps \le 1$ is necessary in Theorem ~\ref{thm:BanAlg}. 

\smallskip
\subsection{NSS sequences in Banach-Lie groups and groups with a minimal metric} 
A rather easy consequence of Theorem \ref{thm:BanAlg} is that Jamison sequences are NSS for the class of \emph{linear} Banach-Lie groups. In this paper, a real (resp.\ complex) \emph{linear Banach-Lie group} $G$ is a topological group for which there exists an injective continuous homomorphism from $G$ into the group 
of invertible elements of a real (resp.\ complex) Banach algebra $A$. Corollary \ref{Cor:Lie1} below follows directly from Theorem \ref{thm:BanAlg} above when one considers complex linear Banach-Lie groups. The same result holds true for real linear Banach-Lie groups, by considering the complexification of real Banach algebras as in \cite{BonDun}*{p.\,68}.

\begin{corollary}[Characterizing {NSS} sequences for linear Banach-Lie groups]\label{Cor:Lie1} 
Let $(n_k)_{k\ge 0}$ be a sequence of integers with $n_0=1$. The following assertions are equivalent:
\begin{itemize}
\item[(i)]  \nk is NSS for the class of linear Banach-Lie groups;

\item[(ii)] \nk is a Jamison sequence.

\end{itemize}
\end{corollary}

We do not know whether Corollary \ref{Cor:Lie1} can be extended to the class of Banach-Lie groups. See Remark \ref{pasencoreecrite} after the proof of Theorem \ref{Thm:Lie1} for a discussion of the difficulties that arise when considering Banach-Lie groups instead of (finite dimensional) Lie groups.

\begin{problem}\label{pro:open}
Let \nk be a Jamison sequence with $n_0=1$. Is \nk NSS for the class of Banach-Lie groups?
\end{problem}

Another interesting class of topological groups which has been recently introduced and studied by Rosendal in \cite{Rosendal} is that of groups possessing a minimal metric.

\begin{definition}[minimal metric; \cite{Rosendal}]\label{def:min}
A  metric $d$ on a (metrizable) topological group $G$ is said to be 
\emph{minimal} if it is compatible with the topology of $G$, left-invariant (that is, $d(hg,hf) = d(g,f)$ for all $g$, $f$, $h$ in $G$) and, for every other compatible left-invariant metric $\partial$ on $G$, the map
$$
{\rm id}\colon (G,\partial)\mapsto (G,d)
$$
is Lipschitz in a neighborhood of the identity $e$ of $G$, \emph{i.e.}\,, there is a neighborhood $U$ of $e$ and a positive constant $K$ such that for all $g,f\in U$,
$$
d(g,f)\leq K\cdot \partial(g,f).
$$
\end{definition}

Remark that minimal metrics coincide with metrics which are termed \emph{weak Gleason} in \cite{Tao5} as they underlie Gleason's results in \cite{Gleason}; see \cite{Rosendal} and \cite{Tao5}. It should be also noted that groups with minimal metrics have no small subgroups (\cite{Rosendal}*{p.\,198}) and thus locally compact groups with minimal metrics are isomorphic to Lie groups. It comes as a natural question to ask:

\begin{problem}\label{pro:open2}
Let \nk be a Jamison sequence with $n_0=1$.
Is \nk NSS for the class of groups with a minimal metric?
\end{problem}

The following result provides a partial answer to Problems \ref{pro:open} and \ref{pro:open2} for sequences with bounded quotients. 

\begin{theorem}[Sequences with bounded quotients as NSS sequences]\label{thm:Lie2}
Let $(n_k)_{k\ge 0}$ be a sequence of integers with $n_0 = 1$ and 
$$\sup_{k\ge 0} \frac{n_{k+1}}{n_k} \le c < +\infty .$$ 
Then \nk is NSS for the class of Banach-Lie groups and 
for the class of topological groups possessing a minimal metric.
\end{theorem}

\subsection{Sequences characterizing positive operators}\label{subsec:1.5}
Jamison sequences also appear naturally in characterizations of (self-adjoint) positive operators by the accretiveness of some of their powers.
It has been proved by Johnson~\cite{John} for matrices and by DePrima and Richard~\cite{DPR} for operators that a bounded linear operator $A $ on a complex Hilbert space $H$ is a (semi-definite) positive operator if and only if all iterates $A^n$, $n\ge 0$, are accretive. See also \cites{Laub,Uch,GWZ} for related results. Recall that an operator $B$ is said to be \emph{positive} (we write $B\ge 0$) if $\pss{Bx}{x} \ge 0$ for every $x\in H$. The operator $B$ is \emph{accretive} if $\Re B \ge 0$, where $\Re B = (B+B^*)/2$. 
We also write $B > 0$ when $B$ is positive and invertible; notice that this is equivalent to $B \ge \eps I$ for some positive $\eps$. 
It is also true \cite{DPR} that $A > 0$ if and only if $\Re A^n > 0$ for all $n \ge 0$. 
Concerning subsequences, Shiu proved in \cite{Shiu} that powers of $2$ suffice: if $\Re A^{2^k} \ge 0$ for every $k\ge 0$, then $A\ge 0$. 
The same proof show that $A > 0$ whenever $\Re A^{2^k} > 0$ for every $k\ge 0$.

The reader should not be surprised now that we ask which sequences can replace the powers of $2$ in Shiu's result. The answer is obtained in the following theorems, which give all admissible sequences characterizing positive invertible operators by the accretiveness of some of their powers. We start with the characterization of positive invertible operators.

\begin{theorem}[Sequences characterizing positive invertible operators]\label{thm:DePrimaIN}
Let $(n_k)_{k\ge 0}$ be a sequence of integers with $n_0 = 1$. The following assertions are equivalent:
\begin{itemize}
\item[(a)] Every Hilbert space operator $A$ for which $\Re A^{n_k} > 0$ for every $k\ge 0$ is a positive invertible operator;

\item[(b)] Every complex number $c$ with $\Re c^{n_k} > 0$ for every $k\ge 0$ is real and satisfies $c > 0$;

\item[(c)] The pair $(\nk,\sqrt{2})$ is a Jamison pair.
\end{itemize}
\end{theorem}

We also have the following variant for positive operators. 

\begin{theorem}[Sequences characterizing positive operators]\label{thm:DePrima}
Let $(n_k)_{k\ge 0}$ be a sequence of  integers with $n_0 = 1$. The following assertions are equivalent:
\begin{itemize}
\item[(a)] Every Hilbert space operator $A$ for which $\Re A^{n_k} \ge 0$ for every $k\ge 0$ is a positive operator;

\item[(b)] Every complex number $c$ with $\Re c^{n_k} \ge 0$ for every $k\ge 0$ is real and satisfies $c \ge 0$;

\item[(c)] Every $\ld \in \T$ such that $\sup_{k\ge 0} \abs{\ld^{n_k}-1} \le \sqrt{2}$ satisfies $\ld = 1$.
\end{itemize}
\end{theorem}

\begin{remark}
We obtain from Theorems \ref{thm:DePrima} and \ref{thm:DePrimaIN} the amusing consequence that the following operator-theoretical implication holds true. If the assertion ``Every Hilbert space operator $A$ for which $\Re A^{n_k} \ge 0$ for all $k\ge 0$ is a positive operator'' is true, then ``Every Hilbert space operator $A$ for which $\Re A^{n_k} > 0$ for all $k\ge 0$ is a positive invertible operator'' is also true. Indeed, the (c) of Theorem~\ref{thm:DePrima} is stronger than the (c) of Theorem~\ref{thm:DePrimaIN}. We are not aware of a simple argument proving directly this operator-theoretical implication. 
\end{remark}

\begin{remark} 
Notice also that there are sequences satisfying the equivalent conditions of Theorem \ref{thm:DePrimaIN} which do not satisfy the conditions of Theorem \ref{thm:DePrima}. Consider for instance the sequence $n_k= 3^k$ for $k\ge 0$. Then, by Proposition~\ref{prop:bddq}, \nk is a Jamison sequence with Jamison constant $2\sin(\pi/4) = \sqrt{2}$. Thus the sequence of powers of $3$ satisfies the conclusions of Theorem ~\ref{thm:DePrimaIN}. On the other hand, the sequence $n_k= 3^k$ does not satisfy condition (b) of Theorem ~\ref{thm:DePrima}: we have $\Re i^{3^k} = 0$ for every $k\ge 0$. 
\end{remark}

\begin{remark}
Recall the following result proved by Shiu in \cite{Shiu}, and which has been generalized in Theorems \ref{thm:DePrima} and \ref{thm:DePrimaIN}: for a Hilbert space bounded linear operator $A$, if $\Re A^{2^k} \ge 0$ for every $k\ge 0$, then $A\ge 0$. Shiu's result is very sensitive to perturbations, in the sense that replacing one term into the sequence of powers of $2$ can destroy the  property above. Indeed, consider the sequence $(m_k)_{k\ge 0}$ whose terms are given by
$ 1, \blue{3}, 4, 8, 16, \ldots  $,
which is obtained by replacing $2$ by $3$ on the second place of the sequence of powers of $2$.  
The sequence $(m_k)_{k\ge 0}$ does not satisfy the analogue of Shiu's result. Indeed, we have~$\textrm{ Re } (i^{m_k}) \ge 0$ for every $k\ge 0$. The reason is that replacing one term of the sequence can drastically change the Jamison constant. Observe that the Jamison constant of the sequence $(m_k)_{k\ge 0}$ is $\sqrt{2}$ (while the Jamison constant of the sequence $(2^k)_{k\ge 0}$ of powers of $2$ is $\sqrt{3})$. Indeed, suppose that $\lambda\in\T$ is such that $\sup_{k\ge 0}|\lambda^{m_k}-1|<\sqrt{2}$. Then $\sup_{k\ge 0}|(\lambda^{4})^{2^k}-1|<\sqrt{2}<\sqrt{3}$, so that $\lambda^{4}=1$. Since $|\lambda-1|<\sqrt{2}$, $\lambda=1$. Hence $((m_k)_{k\ge 0},\sqrt{2})$ is a Jamison pair. Since $\sup_{k\ge 0}|i^{m_k}-1|=\sqrt{2}$, $\sqrt{2}$ is the Jamison constant of $(m_k)_{k\ge 0}$. A recent note dealing with perturbations of Jamison sequences is \cite{Paulos}.
\end{remark}

\section{Jamison sequences in normed algebras}\label{Sec3}
Our aim in this section is to prove Theorem \ref{thm:BanAlg}, as well as some related results and consequences, including in particular several spectral characterizations of Jamison sequences in normed algebras (see Theorem \ref{thm:spectral} below).
Given an element $a$ of a Banach algebra $A$, we denote by $\sigma(a)$ the \emph{spectrum} of the element $a$, and by $r(a)$ its \emph{spectral radius}. We write 
$\overline{\mathbb{D}}=\{\lambda\in\T\textrm{ ; }|\lambda|\le 1\}$
and $\T=\{\lambda \in\C\,;\,|\lambda |=1\}$.

 We begin with the proof of Theorem \ref{thm:BanAlg}.

\subsection{Proof of Theorem \ref{thm:BanAlg}}
It is clear that if $(A, (n_k)_{k\ge 0}, \varepsilon)$ is NSS for any complex normed algebra $A$, then, in particular, $(\C, (n_k)_{k\ge 0}, \varepsilon)$ is also NSS.  

Suppose now that $(\C, (n_k)_{k\ge 0}, \varepsilon)$ is NSS. Then $z=1$ is the only complex number satisfying $\sup_{k\ge 0} |z^{n_k} - 1| < \eps$. In particular, $((n_k)_{k\ge 0},\eps)$ is a Jamison pair. By considering $z=0$ we obtain that $\eps \le 1$. 

Suppose now that $(n_k)_{k\ge 0}$ is a Jamison sequence and let $\eps \le 1$ be such that $((n_k)_{k\ge 0},\eps)$ is a Jamison pair.  We undertake the proof that $(A, (n_k)_{k\ge 0}, \varepsilon)$ is NSS for any complex normed algebra $A$. By considering the completion of the normed algebra, we can assume without loss of any generality that $A$ itself is a Banach algebra. Suppose that $a\in A$ satisfies 
 \begin{equation}\label{eq:4.1}
 \|a^{n_k} - e\| < \eps
\end{equation}
 for every $k\ge 0$. This implies that $\|a^{n_k}\| \le 1 + \eps$ and thus the spectral radius of $a$
 satisfies
 $ r(a) = \lim_{k\to\infty} \|a^{n_k}\|^{1/n_k} \le 1.$ So $\sigma(a) \subset \overline{\D}$.
 Since $n_0 = 1$ and $\epsilon \le 1$, the equation \eqref{eq:4.1} for $k=0$ implies that $a$ is invertible. We have 
 \begin{equation}\label{eq:here}
 \|a^{-n_k}\| \le \|a^{-n_k} - e\| + 1 
       = \| a^{-n_k}(e-a^{n_k})\| + 1 
        \le \| a^{-n_k}\|\eps + 1.
 \end{equation}
 Therefore 
 $ \|a^{-n_k}\| \le \frac{1}{1-\eps}$
 which implies that $r(a^{-1})\le 1$ and $\sigma(a^{-1}) \subset \overline{\D}$. Thus $r(a) = 1$ and $\sigma(a) \subset \T$. Let $\lambda \in \sigma(a)$. Then $|\lambda|=1$ and $\lambda^{n_k}-1 \in \sigma(a^{n_k}-e)$ for every $k$. Hence
 $$|\lambda^{n_k}-1| \le r(a^{n_k}-e) \le \| a^{n_k}-e\| < \eps,$$
 for every $k$. Since $((n_k)_{k\ge 0},\eps)$ is a Jamison pair, we obtain that $\lambda = 1$. 
 Thus $\sigma(a) = \{1\}$.

In a Banach algebra it is possible to define the logarithm (principal branch) of some elements $x \in A$ by the holomorphic functional calculus. In particular, for $x\in A$ with $\|x-e\| < 1$ we have
$$ \ln (x) = \sum_{j=1}^{\infty} \frac{(-1)^{j-1}}{j}(x-e)^j .$$
As $\|a^{n_k} - e\| < \eps$, we can write 
$$\ln (a^{n_k}) = \sum_{j=1}^{\infty} \frac{(-1)^{j-1}}{j}(a^{n_k}-e)^j $$
and thus 
\begin{align}\label{eq:est}
\norme{\ln (a^{n_k})} &\le \sum_{j=1}^{\infty} \frac{1}{j}\norme{(a^{n_k}-e)^j} \le \sum_{j=1}^{\infty} \frac{1}{j}\norme{a^{n_k}-e}^j 
     \le \sum_{j=1}^{\infty} \eps^j \le \frac{\eps}{1-\eps}
\end{align}
for every $k$. The principal branch of the logarithm satisfies the identity 
$$ \ln (z^j) = j\ln (z) \text{ whenever } -\frac{\pi}{j} < \Arg (z) \le \frac{\pi}{j}$$ where $\Arg (z)\in (-\pi,\pi)$.
Since $\sigma(a) = \{1\}$, we have $\ln (z^{n_k}) = n_k\ln (z)$ in a neighborhood of the spectrum of $a$. 
By the classical properties of the holomorphic functional calculus in a Banach algebra we have $\ln (a^{n_k}) = n_k\ln (a)$. Therefore, using \eqref{eq:est}, 
$$ n_k\norme{\ln (a)} = \norme{\ln (a^{n_k})} \le \frac{\eps}{1-\eps}$$
for every $k$. Thus $\ln (a) = 0$. Denoting by $\exp$ the exponential function, we have $\exp(\ln x) = x$ whenever $\norme{x-e} < 1$. Therefore $a = \exp(0) = e$.
This proves that $(A, (n_k)_{k\ge 0}, \varepsilon)$ is NSS.
\qed

\subsection{Explicit constants for sequences with bounded quotients}
A generalization of the result of Chernoff quoted in the introduction has been proved by Gorin in \cite{Gor}. It runs as follows: suppose that $A$ is a unital normed algebra, $0<\varepsilon < 1$ and the sequence $(n_k)_{k\ge 0}$ verifies
$$ \frac{n_{k+1}}{n_k} < \frac{\pi - \arcsin(\varepsilon/2)}{\arcsin(\varepsilon/2)}\quad\textrm{for every }k\ge 0.$$
Then $(A, (n_k)_{k\ge 0}, \varepsilon)$ is NSS. 
See also the paper \cite{KMOT} by Kalton, Montgomery-Smith, Oleszkiewicz and Tomilov.
As a consequence of Theorem \ref{thm:BanAlg}, we retrieve Gorin's result, as well as the following variant which was stated (in a slightly different form) and proved in \cite{KMOT}*{Corollary\,4.2}:

\begin{corollary}[explicit constants; \cite{KMOT}]\label{cor:kmot}
Let $A$ be a complex normed algebra and let $\nk$ be an increasing sequence of positive integers with $n_0 = 1$ and 
$$\sup_{k\ge 0} \frac{n_{k+1}}{n_k} \le c.$$ 
Then the triplet $(A,\nk,\min(2\sin(\pi/(c+1)),1))$ is NSS for complex normed algebras. Thus the triplet $(A,\nk,2\sin(\pi/(c+1)))$ is NSS for $c\ge 5$ and $(A,\nk,1)$ is NSS whenever $c < 5$.
\end{corollary}

\begin{proof}
The proof follows from Theorem \ref{thm:BanAlg} and Proposition~\ref{prop:bddq}.
\end{proof}

\subsection{Spectral characterizations of Jamison sequences}
 The following result, which is in part a strength\-ening of Theorem~\ref{thm:BanAlg}, provides other spectral characterizations of Jamison sequences. If $T$ is a bounded operator on a complex  Banach space $X$, we denote by $\sigma(T)$ the spectrum of $T$, and by  $\sigma_p(T)$ its \emph{point spectrum} (\emph{i.e.} the set of its  eigenvalues). The set $\sigma_p(T) \cap \T$ of all eigenvalues of modulus $1$ of $T$ is called the \emph{unimodular point spectrum} of $T$.

\begin{theorem}[spectral characterizations]\label{thm:spectral}
 Let $(n_k)_{k\ge 0}$ be a  sequence of  integers with $n_0 = 1$. 
 The following assertions are equivalent:
 \begin{itemize}
 \item[(i)] \nk is a Jamison sequence;
 
\item[(ii)] There exists $\eps \in ]0,1]$ such that for every complex normed algebra $A$ with unit $e$ we have 
$$ \sup_{k\ge 0}\|a^{n_k} - e\| < \eps \Longrightarrow a = e;$$
 
  \item[(iii)] There exists $\eps \in ]0,1]$ such that for every complex Banach algebra $A$ with unit $e$ we have 
 $$ \sup_{k\ge 0}\|a^{n_k} - e\| < \eps \Longrightarrow \sigma(a) \textrm{ is countable;}$$
 
   \item[(iv)] There exists $\eps \in ]0,1]$ such that for any bounded operator $T$ on a complex separable Banach space $X$, we have 
   $$ \sup_{k\ge 0}\|T^{n_k} - I\| < \eps \Longrightarrow \sigma(T) \textrm{ is countable;}$$
   
      \item[(v)] There exists $\eps \in ]0,1]$ such that for any bounded operator $T$ on a complex separable Hilbert space $H$, we have
   $$ \sup_{k\ge 0}\|T^{n_k} - I\| < \eps \Longrightarrow \sigma(T) \textrm{ is countable;}$$
    
   \item[(vi)]  There exists $\eps \in ]0,1]$ such that for any bounded operator $T$ on a complex separable Hilbert space $H$, we have 
 $$ \sup_{k\ge 0}\|T^{n_k} - I\| < \eps \Longrightarrow \sigma_p(T) \textrm{ is countable;}$$

   \item[(vii)]  There exists $\eps \in ]0,1]$ such that for any bounded operator $T$ on a complex separable Hilbert space $H$, we have 
 $$ \sup_{k\ge 0}\|T^{n_k} - I\| < \eps \Longrightarrow \sigma_p(T)\cap \T \textrm{ is countable}.$$

 \end{itemize}
 \end{theorem}

\subsection{Proof of Theorem \ref{thm:spectral}}
Let \nk be a Jamison sequence and let $\eps \in ]0,1]$ be such that $(\nk,\eps)$ is a Jamison pair. It was proved in Theorem \ref{thm:BanAlg} that $(A, (n_k)_{k\ge 0}, \varepsilon)$ is NSS for any complex normed algebra $A$. This shows the implication (i) $\Rightarrow$ (ii). The implications 
(ii)$\Rightarrow$ (iii) $\Rightarrow$ (iv) $\Rightarrow$ (v) $\Rightarrow$ (vi) $\Rightarrow$ (vii)
are obvious. 

Suppose now that  (vii) holds true for some constant $\varepsilon\in (0,1]$, that is,  for every  $T\in\mathcal{B}(H)$ we have that
 $$ \sup_{k\ge 0}\|T^{n_k} - I\| \le \varepsilon \Longrightarrow \sigma_p(T) \textrm{ is countable.}$$ 
 We want to show that \nk is a Jamison sequence; the proof is a modification of a construction in \cite{EiGr}.
Suppose, for the sake of contradiction, that \nk is not a Jamison sequence. Denote by $(e_{n})_{n\ge 1}$ the canonical basis of $\ell_{2}(\N)$. It is proved in \cite{EiGr}*{Theorem\,2.1} that there exists an operator $T$ on $\ell_{2}(\N)$ with uncountable unimodular point spectrum such that $\sup_{k\ge 0}||T^{n_{k}}||<+\infty$. More precisely, $T$ has the form $T=D+B$, where $D$ is a diagonal operator and $B$ is a weighted backward shift with respect to the basis $(e_{n})_{n\ge 1}$. We have $De_{n}=\lambda _{n}e_{n}$ for every $n\ge 1$, where the $\lambda _{n}$'s are distinct complex numbers with $|\lambda _{n}|=1$, and $Be_{1}=0$, $Be_{n}=\alpha _{n-1}e_{n-1}$, $n\ge 2$, where the $\alpha _{n}$'s are certain positive weights. The diagonal coefficients $\lambda _{n}$ are chosen using the fact that $(n_{k})_{k\ge 0}$ is a non-Jamison sequence, and belong to a perfect compact subset $K$ of $\T$.  
Notice that what we denote here by $K$ is called $K'$ in \cite{EiGr} and that it is proved in \cite{EiGr} that this subset is separable for the metric on $\T$ defined by 
$d_{(n_{k})}(\lambda ,\mu )=\sup_{k\ge 0}|\lambda ^{n_{k}}-\mu ^{n_{k}}|$, $\lambda,\mu\in\T$.

 Let $\varepsilon>0$ be an arbitrarily small  number. Although it is not used in the proof of \cite{EiGr}*{Theorem\,2.1}, one can suppose without loss of generality that $d_{(n_{k})}(\lambda ,1 )<\varepsilon/2 $ for every $\lambda \in K$. 
Thus $\sup_{k\ge  0}|\lambda_{n} ^{n_{k}}-1|<\varepsilon/2 $ for every $n\ge 1$, so that
$\sup_{k\ge 0}||D^{n_{k}}-I||<\varepsilon /2$. Moreover, the proof of \cite{EiGr}*{Theorem\,2.1} shows that the construction can be carried out in such a way that 
$\sup_{k\ge 0}||T^{n_{k}}-D^{n_{k}}||<\varepsilon/2 $. 
Putting these two estimates together yields that $\sup_{k\ge 0}||T^{n_{k}}-I||<\varepsilon$.
\par\medskip
We have thus shown that if $(n_{k})_{k\ge 0}$ is not a Jamison sequence, there exists  for every $\varepsilon >0$ a Hilbert space operator $T$ with $\sigma _{p}(T)\cap\T$ uncountable such that $\sup_{k\ge 0}||T^{n_{k}}-I||<\varepsilon $. This contradiction shows that $(n_{k})_{k\ge 0}$ has to be a Jamison sequence.

\subsection{Jamison sets of real numbers}\label{sect:inR}
Some of the results presented until now have analogues for Jamison sequences (or sets) in other semigroups than $\N$. We present here a consequence of Theorem~\ref{thm:BanAlg} for Jamison sets in $[0,+\infty)$. 

\begin{definition}
Let $E$ be a set of nonnegative real numbers. We say that $E$ is a \textit{Jamison set} in $\R$ if for every separable complex Banach space $X$ and for every $C_0$-semigroup $(T_t)_{t\ge0}$ of bounded linear operators on $X$ (with infinitesimal generator $A$) which is partially bounded with respect to the set $E$, that is $\sup_{t\in E}\Vert T_t\Vert<+\infty$, the set $\sigma_p(A)\cap i\mathbb{R}$ is countable.
\end{definition}

Observe that a Jamison subset of $[0,+\infty)$ is necessarily unbounded. 
We may, without any loss of generality, suppose that $0\in E$. Since a Jamison set is unbounded, $E$ contains a nonzero element. By dividing with this number we can, without any loss of generality, suppose also that $1\in E$. We recall now a characterization, obtained in \cite{Dev}*{Lemma 3.8 and Theorem 3.9},
of Jamison sets of positive real numbers.
Let $F = \{\lfloor t\rfloor : t\in E\}$, where $\lfloor x\rfloor$ denotes the largest integer less or equal to the real number $x$. Since $\{0,1\} \subset E$, both $F\setminus \{0\}$ and $F+1 := \{f+1 : f\in F\}$ are sets of positive integers containing $1$. It therefore makes sense to speak about Jamison sets in $\N$ for these two sets: we say for instance that $F+1$ is a Jamison set (in $\N$) if the strictly increasing sequence of its elements, $(n_k)_{k\ge0}$, starting from $n_0=1$, is a Jamison sequence.

\begin{theorem}[characterization of Jamison sets in $\R$; \cite{Dev}]\label{JamisonR}
Let $E\subset [0,+\infty)$ be a set of real numbers such that $\{0,1\}\subset E$. The following assertions are equivalent$:$
\begin{enumerate}
\item[$(1)$] $E$ is a Jamison set in $\R;$
\item[$(2)$] $F\setminus \{0\}$ is a Jamison set in $\N;$
\item[$(2)$] $F+1$ is a Jamison set in $\N$.
\end{enumerate}
\end{theorem}

Our aim is now to prove the following result, which gives a strong property of Jamison subsets $E \subset [0,+\infty)$ which contain a neighborhood of $0$ in $[0,+\infty)$. Without any loss of generality we can suppose that $[0,1] \subset E$.

If $E$ is a Jamison set with $[0,1] \subset E \subset [0,+\infty)$, let $0<\varepsilon_{F}\le 1$ denote a Jamison constant for the sequence $\nk$, with $F+1 = \{1,n_1,n_2, \cdots\}$. We have the following result:

\begin{theorem}[Jamison sets and $C_0$-semigroups]\label{characterizationR}
Suppose that $[0,1] \subset E \subset [0,+\infty)$ and that $E$ is a Jamison set in $\R$.
Let $\varepsilon>0$ be such that $\varepsilon <\min\left(\varepsilon_{F}/3,1/3\right)$. Let $X$ be a  complex Banach space and let $(T_t)_{t\ge0}$ be a $C_0$-semigroup of bounded linear operators on $X$ such that 
\begin{equation}\label{condition}
\Vert T_t-{I}\Vert\le \varepsilon \quad\textrm{ for every } t\in E.
\end{equation}
Then $T_t={I}$ for every  $t\ge0$.
\end{theorem}

\begin{proof}[Proof of Theorem~\ref{characterizationR}] 
Since $[0,1]$ is a subset of $E$, \eqref{condition} implies that $\sup_{t\in [0,1]}\Vert T_t-I\Vert\le \varepsilon$. By the triangle inequality,  $\sup_{t\in [0,1]}\Vert T_t\Vert \le 1+\varepsilon \le 2$.
Using the notation $\{t\} = t-\lfloor t\rfloor$ for the fractional part of $t$, we have $\Vert T_{1-\{t\}}\Vert \le 2$ for every $t\ge 0$. Indeed, we have $1-\{t\} \in [0,1]$. Also, we have 
\begin{align*}
\Vert T_{\lfloor t\rfloor+1}-I\Vert=\Vert T_{t+1-\{t\}}-I\Vert &=\Vert T_{1-\{t\}} (T_{t}-I) + T_{1-\{t\}} - I\Vert\\
&\le \Vert T_{1-\{t\}}\Vert\,.\, \Vert T_t-I\Vert+\Vert T_{1-\{t\}}-I\Vert\le 3\varepsilon
\end{align*}
for every $t\ge0$. So
$
\Vert T^n_{1}-I\Vert = \Vert T_n-I\Vert\le 3\varepsilon
$  for every $n\in F+1$, where $F=\big\{\lfloor t\rfloor\, : \,t\in E\big\}$. Since $F+1$ is a Jamison set in $\N$ and $3\eps <\min\left(\varepsilon_{F},1\right)$, we obtain $T_1 = I$ by Theorem~\ref{thm:BanAlg}. Therefore $T_n = I$ for each $n\in \N$ and $T_t = T_{\{t\}}$ for every $t\ge0$. Since $\sup_{t\in [0,1]}\Vert T_t-I\Vert\le \varepsilon$, it follows that $\sup_{t\ge0}\Vert T_t-I\Vert\le \varepsilon$. We deduce that, for any fixed $t\ge0$, we have $\sup_{k\ge0}\Vert T_{t}^{k+1}-I\Vert<\varepsilon $.
Since $((k+1)_{k\ge 0},1)$ is a Jamison pair, it follows from Theorem \ref{thm:BanAlg} that $T_{t}=I$ for every $t$.
\end{proof}

 \section{NSS sequences in Lie groups}\label{Sec4}
We begin this section with the proof of Theorem \ref{Thm:Lie1}.

\begin{proof}[Proof of Theorem \ref{Thm:Lie1}]
Suppose  that \nk is a Jamison sequence with $n_0=1$. According to Corollary~\ref{Cor:Lie1}, which follows directly from Theorem~\ref{thm:BanAlg}, the sequence \nk is NSS for the class of linear Banach-Lie groups. We now want to show that \nk is NSS for a Lie group $G$, that is, that there exists an open neighborhood $V$ of the identity element $e$ of $G$ such that, if $g^{n_k} \in V$ for every $k\ge 0$, then $g=e$. Without loss of any generality we can assume that $G$ is connected. We denote by $\mathrm{ Ad } : G \mapsto GL(E)$ the adjoint representation of $G$ on $E= T_e(G)$, the tangent space at $e$ of $G$. Using Theorem~\ref{thm:BanAlg}, we obtain the existence of an open neighborhood $W$ in $GL(E)$ of the identity $I$ such that, for every $M\in GL(E)$, the following implication holds true: if
 $M^{n_k} \in W$ for every $k\ge 0$, then $ M=I$.
 Consider 
 $$W' = Ad^{-1}(W) = \{g \in G : Ad(g) \in W\},$$
 which is an open neighborhood of $e$ in $G$. An element $g\in G$ satisfying $g^{n_k}\in W'$ for every $k$ has the property that $\mathrm{Ad}(g^{n_k}) \in W$ for every $k$. As the adjoint representation is a group morphism, we have that $M:= \mathrm{Ad}(g)$ satisfies $M^{n_k} \in W$ for every $k$, and thus $\mathrm{Ad}(g) = M = I$. As the kernel of the adjoint representation is the center of the group, $\mathrm{Z}(G)$, we infer that $g \in \mathrm{Z}(G)$ (recall that the \emph{center} of $G$ is the set of all elements commuting with all elements of the group $G$). 
 
 The center $\mathrm{Z}(G)$ is an abelian locally compact group and it is a closed subgroup of the Lie group $G$. Therefore, $\mathrm{Z}(G)$ is an abelian Lie group. Consider the connected component $G_0$ of $\mathrm{Z}(G)$ containing the identity element $e$. 
It is a connected, abelian Lie group, and it is hence isomorphic, as a Lie group, to a group of the form $\K^{n}/\Gamma $ where $\Gamma $ is a lattice in $\K^{n}$. Here, $\K=\R$ or $\C$, depending on whether $G$ is a real or complex Lie group. The lattice $\Gamma $ has the form $\Gamma =\Z u_{1}+\cdots+\Z u_{r}$, where $u_{1},\dots,u_{r}$ are $\R$-independent vectors in $\K^{n}$ and $0\le r\le n$ if $\K=\R$, while $0\le r\le 2n$ if $\K=\C$.
\par\medskip
Let $||\,.\,|| $ be a norm on $\K^{n}$, and let $\varepsilon >0$. Let $E=\textrm{span}_{\R}\,[u_{1},\dots,u_{r}]$, and let $F\subseteq \K^{n}$ be a 
real subspace
such that $\K^{n}=E\oplus F$. Let $x\in \K^{n}$, which we write as $x=u+v=\sum_{i=1}^{r}\alpha _{i}u_{i}+v$, $u\in E$, $v\in F$, $\alpha _{1},\dots,\alpha _{r}\in \R$. Suppose that 
$\textrm{dist}(n_{k}x,\Gamma )=\inf_{\gamma \in\Gamma }||n_{k}x-\gamma ||<\varepsilon $
for every $k\ge 0$. Denote by $P$ the projection of $\K^{n}$ on $F$ along $E$. We have
$n_{k}||v|| = ||n_{k}Px||\le ||P||\varepsilon $ for every $k$. Hence $v=0$, and so $x=u\in E$. Then
\[
\inf_{a_{1},\dots,a_{r}\in\Z}\,\,\,\Bigl |\!\Bigl |\, \sum_{i=1}^{r}(n_{k}\alpha _{i}-a_{i})u_{i}\,\Bigr | \!\Bigr|<\varepsilon 
\]
for every $k\ge 0$. For every $i=1,\dots,r$, let $P_{i}$ denote the projection of $E$ onto the span of the vector $u_{i}$ along the space $\textrm{span}\,[u_{j}\,;\,j\neq i]$. Then $\inf_{a_{i}\in\Z}|n_{k}\alpha _{i}-a_{i}|<\varepsilon ||P_{i}||$. If $\varepsilon >0$ is so small that $\varepsilon 
||P_{i}||$ is less than the Jamison constant of the sequence $(n_{k})_{k\ge 0}$ for every $i=1,\dots,r$, we get that $\alpha _{i}\in\Z$ for every $i$, i.e.\ that $u\in\Gamma $. Hence the class of $x$ in the quotient $\K^{n}/\Gamma $ is $0$, \emph{i.e.}\ $(n_{k})_{k\ge 0}$ is NSS for the group $\K^{n}/\Gamma $. We have thus proved that Jamison sequences are NSS for connected abelian Lie groups, in particular for $G_{0}$.

 Let $V_0$ be an open neighborhood of $e$ in $G_0$, of the form $V_0 = V_1\cap \mathrm{Z}(G)$, with $V_1$ an open neighborhood of $e$ in $G$, such that, for every $g\in \mathrm{Z}(G)$, we have the following implication: if
 $g^{n_k} \in V_0$ for every $k$, then $g = e$.
 Let $V = V_1 \cap W'$ which is an open neighborhood of $e$. If $g\in G$ is such that $g^{n_k} \in V$ for every $k$, then $g^{n_k} \in \mathrm{Z}(G)\cap V_1 = V_0$, and so $g=e$.
\end{proof}

\begin{remark}\label{pasencoreecrite}
We do not know whether the proof above can be generalized to show that Jamison sequences are {NSS} for Banach-Lie groups. The theory of Banach-Lie groups differs from that of (finite-dimensional) Lie groups in several aspects. For instance, contrary to the case of Lie groups, closed subgroups of Banach-Lie  groups are not necessarily Banach-Lie groups (see for instance \cite{HofmannMorrisEMS}*{p.\,110}). 
Also, the description of connected abelian Lie groups used above is specific to Lie groups (see \cite{Neeb} and \cite{MT} for extensions to much more general settings). 

It should be also be mentioned that, according to \cite{LumVal} (see also \cite{BeNe}), a connected (finite-dimensional, real) Lie group $G$ has a continuous faithful embedding into the group of invertible elements of some Banach algebra with
its norm topology if and only if $G$ is a linear Lie group.
\end{remark}

\begin{remark}
We present here an elementary proof, using the Jordan canonical form of a matrix, that if $\nk$ is a Jamison sequence, then \nk is NSS for the class of compact Lie groups. Let $\eps \in (0,1)$ be such that $(\nk, \eps)$ is a Jamison pair. It is known (see for instance \cite{Segal}*{Section\,II}) that every compact Lie group is a matrix group. Let $A$ be a $n\times n$ matrix such that $\norme{A^{n_k} - I} \le \eps$ for every $k$. Then $\norme{A^{n_k}}\le 1+\eps$ and thus $\sigma(A)\subset\overline{\mathbb{D}}=\{\lambda\in\T\textrm{ ; }|\lambda|\le 1\}$. An  estimate similar to the one in Equation \eqref{eq:here} shows that $A$ is invertible and that $\sigma(A^{-1})\subset\overline{\mathbb{D}}$. Thus $\sigma(A)\subset\T$. In fact, the only possible eigenvalue of $A$ is $1$. Indeed, if $z$ is an eigenvalue for $A$, with normalized eigenvector $x$, then $\abs{z} = 1$ and 
$ \abs{z^{n_k} - 1} = \abs{z^{n_k} - 1}\|x\| = \norme{A^{n_k}x - x} \le \eps $ for every $k$.
Since $(\nk, \eps)$ is a Jamison pair, we obtain that $z=1$.

 Let now $L$ be an invertible matrix such that $J = L^{-1}AL$ is the Jordan canonical form of $A$, that is $J = L^{-1}AL$ has ones on the diagonal,
zeros and ones on the superdiagonal, and zeros elsewhere. Suppose $k$ is a positive integer between $1$ and $n-1$ such that the $(k,k+1)$ entry of the Jordan canonical form $J$ is $1$. Then, as a simple proof by induction shows, the $(k,k+1)$ entry of $J^p$ is $p$ for every positive integer $p$. This is in contradiction to the fact that the sequence of the norms $(\norme{J^{n_k}})_{k\ge 0}$ is bounded. 
Thus $J = L^{-1}AL$ is the identity matrix, and so the same is true for $A$. Therefore, \nk is NSS for the class of compact Lie groups.

A more direct argument can be given using a particular case of Gelfand's theorem (see \cite{Gel} or \cite{Zem}): a compact Lie group is isomorphic to a subgroup of a unitary group and the only unitary matrix $U$ whose spectrum is the singleton $\{1\}$ is the identity matrix. 
\end{remark}

We now move over to the proof of Theorem \ref{thm:Lie2}. 

\section{Proof of Theorem \ref{thm:Lie2}}\label{Sec5}

\subsection{Proof of Theorem \ref{thm:Lie2} for Banach-Lie groups}
The proof that sequences with bounded quotients are NSS for Banach-Lie groups is a generalization of the classical proof that Lie groups are NSS, see for instance \cite{HofmannMorrisEMS}*{Proposition 2.17}  or \cite{MorrisPestov}*{Theorem\,2.7}; for the convenience of the reader we briefly sketch the argument.
Let $G$ be a Banach-Lie group and let $\mathfrak{g}$ be the corresponding Banach-Lie algebra. Let \nk be a sequence of integers with $n_0 = 1$ and 
$$\sup_{k\ge 0} \frac{n_{k+1}}{n_k} \le c.$$ 
Without loss of generality, we can suppose that $c \ge 2$. 
Let $B$ be an open neighborhood of $0$ in $\mathfrak{g}$ for the Campbell-Hausdorff topology (see \cite{HofmannMorrisEMS} for the definitions) such that 
 there is an exponential function $\exp$ which is an homeomorphism from B onto on open neighborhood $V$ of $e$ in $G$, and is such that whenever
$X\ast Y$ belongs to $B$, $\exp(X\ast Y)=\exp(X)\exp(Y)$.
In other words, $B$ is a local Banach-Lie group with respect to  Hausdorff multiplication. 
Set $U=\exp(\frac{1}{c}B)$, which is an open neighborhood of $e$ in $G$, and let $h\not=e$ belong to $U$. We wish to show that there exists $k$ such that $h^{n_k}$ does not belong to $U$.
Let $\tilde{x}\in \frac{1}{c}B$ be such that $h=\exp(\tilde{x})$. Since $h\not=e$, $\tilde{x}\not = 0$. Let us show that there exists a $k$ such that $n_k \tilde{x}\in B\setminus \frac{1}{c}B$.  We have $n_1 \tilde{x}\in\frac{n_1}{c}B=\frac{n_1}{n_0}\frac{1}{c}B\subset B$. If $n_1 \tilde{x}\not\in
 \frac{1}{c}B$, we are done. Else $n_1 \tilde{x}\in
 \frac{1}{c}B$, and $n_2 \tilde{x}\in
 \frac{n_2}{n_1}\frac{1}{c}B\subset B$, so if $n_2 \tilde{x}\not\in
 \frac{1}{c}B$, we are also done. We continue in this fashion. Since we cannot have $n_k \tilde{x}\in B$ for every $k$, we deduce that there exists a $k$ such that $n_k \tilde{x}\in B\setminus \frac{1}{c}B$. The properties of $B$ and of the exponential function imply that $h^{n_k}=\exp(n_k \tilde{x})\in V\setminus U$, and this proves our claim.

\subsection{Proof of Theorem \ref{thm:Lie2} for groups with a minimal metric}
The proof of Theorem \ref{thm:Lie2} for groups with a minimal metric will be based upon two auxiliary results. Recall that we denote by $e$ the identity element of the group $G$.

\begin{lemma}[\cite{Rosendal}]\label{lem:L1}
Let $G$ be a topological group with a minimal metric $d$. Then there exist $a> 0$ and $K \ge 1$ such that, for $f,\,g\in G$ and any positive integer $n$, the following statements are true:

(1) (\emph{A quantitative NSS condition}) If 
$ \max \{ d(g^i,e) : 1\le i \le n\} < a,$
then
\begin{equation}\label{eq:7.1}
d(g,e) \le \frac{1}{n}\cdot
\end{equation}

(2) (\emph{The weak Gleason property}) If 
$ \max \{ d(g^i,e) : 1\le i \le n\} < a,$
then
\begin{equation}\label{eq:7.2}
nd(g,e) \le Kd(g^n,e).
\end{equation}

(3) (\emph{Multiplication is locally Lipschitz}) If 
$ \max \{ d(f,e), d(g,e)\} < a,$
then
\begin{equation}\label{eq:7.2bis}
d(fg,e) \le K\left( d(f,e) + d(g,e)\right).
\end{equation}
\end{lemma}

\begin{proof}
The results follow from conditions (2) and (3) of \cite{Rosendal}*{Theorem\,3} and from \cite{Rosendal}*{Observation\,10}.
\end{proof}

\begin{lemma}[trapping property]\label{lem:L2}
Let $a$ and $K$ be the constants from Lemma \ref{lem:L1}. Let $b>0$. If $h\in G$, $n$ is a positive integer, and
\begin{equation}\label{eq:7.3}
 \max \{ d(h^i,e) : 1\le i \le n\} < a\quad \textrm{and} \quad
d(h^n,e) < b,
\end{equation}
then
\begin{equation}\label{eq:7.4}
\max \{ d(h^i,e) : 1\le i \le n\} \le \min (a, Kb).
\end{equation}
\end{lemma}

\begin{proof}
Suppose that $h$ is an element of $G$ such that
\eqref{eq:7.3} is true.
Using \eqref{eq:7.2}, we obtain that
\begin{align*}
d(h,e) \le \frac{Kd(h^n,e)}{n} \le \frac{Kb}{n}\le Kb.
\end{align*}
  Suppose that $n\ge 2$.
Since $d$ is a left-invariant metric, we obtain
 \begin{align*}
d(h^2,e) \le d(h^2,h) + d(h,e) = 2 d(h,e) \le  \frac{2Kb}{n}\le Kb.
\end{align*}
A similar proof shows that 
$$ d(h^i,e) \le \frac{iKb}{n}\le Kb$$
for every $1\le i\le n$, which proves Lemma \ref{lem:L2}.
\end{proof}

  \begin{proof}[Proof of Theorem \ref{thm:Lie2} for groups with a minimal metric]
Let $G$ be a group with a minimal metric $d$. Suppose that $n_{k+1}/n_k \le c$. Without loss of any generality we can assume that $c$ is a positive integer. 

Let now $x\in G$ be such that 
\begin{equation}\label{eq:7.5}
\sup \{ d(x^{n_j},e) : j \ge 1\} < \delta := \frac{a}{2K(K+c)}\cdot
\end{equation}
Thus all the $n_j$-powers of $x$ belong to the neighborhood $B_d(e,\delta)$.
We want to show that $x=e$. Set $W= B_d(e,K\delta)$. We will prove by induction the following claim:
 \begin{align}\label{eq:recurr}
\textrm{ For every integer } k \ge 0, \textrm{ all the elements } x, x^2, x^3, \cdots, x^{n_k} \textrm{ belong to } W. 
\end{align}

This is surely true for $k=0$ since $d(x,e) < \delta \le K\delta$. Suppose that \eqref{eq:recurr} is true for a fixed $k\ge 0$ and consider an integer $i$ such that $n_k < i \le n_{k+1}$. Then we can write $i$ as $i = qn_k + r$ for some positive integer $q$ and for some integer $r$ with $0\le r < n_k$. We have 
$$ q = \frac{i-r}{n_k} \le \frac{n_{k+1}-r}{n_k} \le c .$$
Using that multiplication is locally Lipschitz, \emph{i.e.} the condition  (\ref{eq:7.2bis}), and the induction hypothesis, we obtain
 \begin{align*}
d(x^i,e) = d(x^{qn_k}x^r,e) \le K\left( qd(x^{n_k},e)+d(x^r,e)\right) \le K\left(c\delta + K\delta\right) = \frac{a}{2} < a
\end{align*}
(we have $d(x^{qn_k},e) \le qd(x^{n_k},e) <c\varepsilon<a$, so that \eqref{eq:7.2bis} can be applied). 
As a particular case of \eqref{eq:7.5} we have $d(x^{n_{k+1}},e) < \delta$. It follows from the trapping property of Lemma \ref{lem:L2} that 
$$\max \{ d(x^i,e) : 1\le i \le n_{k+1}\} \le K\delta.$$
Therefore all the elements $x, x^2, x^3, \cdots, x^{n_{k+1}}$ belong to $W$. By induction, the property \eqref{eq:recurr} holds for every $k$. Therefore the neighborhood $W$ of $e$ contains all the elements $x^p$, $p\ge 0$. Since a topological group with a minimal metric is NSS by condition (1) of Lemma \ref{lem:L1}, we obtain that $x=e$, which completes the proof. 
   \end{proof}

\begin{remark}
Constructing weak Gleason metrics for NSS locally compact groups is an essential step in proving that every NSS locally compact group is isomorphic to a Lie group. The condition \eqref{eq:7.2} is the so-called escape property of weak Gleason metrics, as introduced in \cite{Tao5}*{p.\,103 ff}. It is proved in \cite{Tao5}*{Theorem\,5.2.1} that, in the locally compact setting, every weak Gleason metric is actually Gleason, meaning that it satisfies a further
estimate on commutators. 
\end{remark}

\section{Proofs of Theorems \ref{thm:DePrimaIN}, \ref{thm:DePrima} and of Corollary \ref{cor:nagisa}}\label{Sec6}

Since the proofs of the two theorems are similar, we only give the proof of Theorem \ref{thm:DePrima}. 

\begin{proof}[Proof of Theorem \ref{thm:DePrima}]
Let us first prove that the statements (b) and (c) of Theorem \ref{thm:DePrima} are equivalent. Suppose that (b) is true, that is, every complex number $c$ with $\Re c^{n_k} \ge 0$ for every $k\ge 0$ satisfies $c \ge 0$. Let $\ld \in \T$ be such that $|\ld^{n_k} - 1| \le \sqrt{2}$ for every $k$. Then 
$2\Re (\ld^{n_k}) = 2 - |\ld^{n_k} - 1|^2 \ge 0$
for every $k$. Using (b), we obtain that $\ld$ is a positive real number. As $|\ld| = 1$, $\ld = 1$. 

Suppose now that (c) is true. Let $c$ be a non-zero complex number such that $\Re c^{n_k} \ge 0$ for every $k\ge 0$. Let $\ld = c/|c|$. Then $|\ld| = 1$ and $\Re (\ld^{n_k}) \ge 0$. Therefore
$|\ld^{n_k} - 1|^2 = 2 - 2\Re (\ld^{n_k}) \le 2.$
Using (c) we get $\ld = 1$ and thus $c = |c|$ is a nonnegative real number. 

Clearly (a) implies (b). In order to show that (b) implies (a), we need to introduce some notation and to recall some results about the numerical range and fractional powers of operators. These results go back to a $1962$ paper by Matsaev and Palant \cite{MaPa}, see also \cites{LiRoSpi,BleWa} and the references therein.
\renewcommand{\qedsymbol}{}
\end{proof}

\begin{proof}[An interlude on fractional powers of operators]
\par\smallskip
Recall that the numerical range of the operator $T$ is the set
$$ W(T) = \left\{ \ps{Tx}{x} : \|x\|=1 \right\} $$
and that the closure of $W(T)$ always contains the spectrum $\sigma(T)$ of $T$. What we will need in the proof that (b) implies (a) are the following results, which follow for instance from Theorems 1.2 and 2.8 from \cite{LiRoSpi} (this paper deals with fractional powers of elements of more general Banach algebras, with or without an involution). Let $T\in \B{H}$ be a Hilbert space operator such that $W(T)$ does not contain any negative real number. Note that, by convexity of the numerical range, this implies that $W(T)$ is included in a certain sector centered in the origin and of opening no greater than $\pi$. Let $m\ge 2$ be an integer. 

(\emph{Existence}) Under the above hypotheses, $T$ has a $m$-root, $S$, in $\B{H}$ such that $S^m = T$ and the numerical range of $S$ lives inside the sector $\Sigma(\pi/m)$ centered in the origin and of opening $2\pi/m$, namely
$$W(S) \subset \Sigma(\pi/m) := \left\{re^{i\theta} : r \ge 0, \, \abs{\theta} \le \frac{\pi}{m}\right\}.$$ 

(\emph{Unicity}) Moreover, if $R\in \B{H}$ is another operator such that $R^m = T$ and $\sigma(R)\subset\Sigma(\pi/m)$, then $R=S$. 
\renewcommand{\qedsymbol}{}
\end{proof}
We shall use the notation $T^{1/m}$ for the unique $m$-root of $T$.
\par\medskip
Let us now go back to the proof of Theorem \ref{thm:DePrima}. 

\begin{proof}[Proof of Theorem \ref{thm:DePrima}, continued] Assume that (b) holds true and let $A$ be a Hilbert space operator such that $\Re A^{n_k} \ge 0$ for all $k\ge 0$. Then the numerical range of $A^{n_k}$ is in the closed right half-plane 
$\overline{\C}_{+} = \{z\in \C : \Re z \ge 0\} .$
In particular, $W(A^{n_k})$ does not contain any negative real number. 

Let $c\in\sigma(A)$. Then, by the spectral mapping theorem, $c^{n_k}\in\sigma(A^{n_k})$, so $c^{n_k}$ belongs to the closure of $W(A^{n_k})$. Thus $\Re c^{n_k} \ge 0$ for every $k$. Therefore (b) implies that $\sigma(A)\subset [0,+\infty)$; in particular, the spectrum of $A$ is included in all the sectors $\Sigma(\pi/(n_k))$.
By the unicity result of fractional powers we have 
$$(A^{n_k})^{1/n_k} = A \quad \textrm{for every } k\ge 0.$$
By the existence part we obtain 
$W(A) \subset \Sigma(\pi/n_k)$ for every $k\ge 0$.
Therefore $W(A)$ is a subset of the positive real axis. Thus $A$ is self-adjoint, and in fact positive since $\Re A \ge 0$. This completes the proof.
\end{proof}

Some generalizations are possible by considering a sequence of sectors, $S(n)$, centered in the origin and replacing (b) by the condition
``$c^{n_k} \in S(n_k)$
for every $k$''. The situation considered in Theorem~\ref{thm:DePrima} corresponds to the case where all the sectors $S(n)$
are the closed right half-plane $\overline{\C}_{+} := \{z\in \C : \Re z \ge 0\}$. This problem, as well as the corresponding vectorial problem (a), have been addressed for matrices in several papers \cites{HaHe,HeSch1,HeSch2}. We decided not to pursue this topic here.

We also obtain the following corollary, generalizing \cite{NagWa}*{Theorem\,1}.

\begin{corollary}\label{cor:nagisa}
Let $(n_k)_{k\ge 0}$ be a sequence of integers with $n_0 = 1$ such that the only complex number $\ld \in \T$ with $\sup_{k\ge 0} \abs{\ld^{n_k}-1} \le \sqrt{2}$ is $\ld = 1$. Then, for a Hilbert space operator $A\in \B{H}$, the following two statements are equivalent:
\begin{itemize}
\item[(i)] $\sup_{k\ge 0} \norme{A^{n_k}-I} \le 1$;
\item[(ii)] $0\le A \le I$.
\end{itemize} 
\end{corollary}

\begin{proof}
If $0\le A \le I$ then $0\le A^n \le I$ for every $n\ge 1$ and thus $\|I-A^n\| \le 1$ for every $n\ge 1$. 

Suppose now that $\|I-A^{n_k}\| \le 1$ for every $k\ge 1$. Let $x\in H$ be a unit vector. We have
\[
\abs{1-\ps{A^{n_k}x}{x}} =  \abs{\ps{(I-A^{n_k})x}{x}} 
 \le \norme{I-A^{n_k}}\norme{x}^2
 \le 1.
\]
This implies that $z= \ps{A^{n_k}x}{x}$ satisfies $\Re(z) \ge \abs{z}^2/2 \ge 0$. Therefore the numerical range of $A^{n_k}$ is in the right closed half-plane for every $k\ge 1$. Using Theorem \ref{thm:DePrimaIN} we obtain that $A$ is (self-adjoint and) positive. 

Let now $\ld \in \sigma(A)$. Then $\ld \ge 0$ and $\ld^{n_k} - 1$ is in $\sigma(A^{n_k} - I)$ for all $k$. Therefore
$ \abs{\ld^{n_k} - 1} \le \|A^{n_k} - I\| \le 1.$
This implies $\abs{\ld^{n_k}} \le 2$ for all $k$ ; hence $\abs{\ld} \le 1$. Thus $\sigma(A) \subset [0,1]$ and $0\le A \le I$.  
\end{proof}

\smallskip

Corollary \ref{cor:nagisa} applies for instance to every Jamison sequence with Jamison constant $<\sqrt{2}$.

\smallskip

\noindent {\bf Acknowledgments.} The revised version of this manuscript was prepared while the first-named author was a member of the Max Planck Institute for Mathematics in Bonn; he would like to thank the MPI for wonderful working conditions. The authors are grateful to the two anonymous referees, who furnished several constructive remarks and numerous suggestions for improving the paper's exposition.


\smallskip

\begin{bibdiv}
  \begin{biblist}
  
  \bib{BaGr}{article}{
   author={Badea, Catalin},
   author={Grivaux, Sophie},
   title={Unimodular eigenvalues, uniformly distributed sequences and linear
   dynamics},
   journal={Adv. Math.},
   volume={211},
   date={2007},
   number={2},
   pages={766--793},
   issn={0001-8708},
   review={\MR{2323544}},
}
  
  \bib{BaGr2}{article}{
   author={Badea, Catalin},
   author={Grivaux, Sophie},
   title={Size of the peripheral point spectrum under power or resolvent
   growth conditions},
   journal={J. Funct. Anal.},
   volume={246},
   date={2007},
   number={2},
   pages={302--329},
   issn={0022-1236},
   review={\MR{2321045}},
}

\bib{BaGrsmf}{article}{
   author={Badea, Catalin},
   author={Grivaux, Sophie},
   title={Sets of integers determined by operator-theoretical properties: Jamison and Kazhdan sets in the group $\mathbb{Z}$},
   conference={
      title={SMF16 : Premier congr\`es de la Soci\'et\'e Math\'ematique de France },
   },
   book={
      series={S\'emin. Congr.},
      volume={31},
      publisher={Soc. Math. France, Paris},
   },
   date={2017},
   pages={39--77},
   review={},
}
  
  \bib{BeidCox}{article}{
   author={Beidleman, J. C.},
   author={Cox, R. H.},
   title={Topological near-rings},
   journal={Arch. Math. (Basel)},
   volume={18},
   date={1967},
   pages={485--492},
   issn={0003-889X},
   review={\MR{0227234}},
}

\bib{BeNe}{article}{
   author={Belti\c t\u a, Daniel},
   author={Neeb, Karl-Hermann},
   title={Finite-dimensional Lie subalgebras of algebras with continuous
   inversion},
   journal={Studia Math.},
   volume={185},
   date={2008},
   number={3},
   pages={249--262},
   issn={0039-3223},
   review={\MR{2391020}},
}

\bib{BleWa}{article}{
   author={Blecher, David P.},
   author={Wang, Zhenhua},
   title={Roots in operator and Banach algebras},
   journal={Integral Equations Operator Theory},
   volume={85},
   date={2016},
   number={1},
   pages={63--90},
   issn={0378-620X},
   review={\MR{3503179}},
}

\bib{BonDun}{book}{
   author={Bonsall, Frank F.},
   author={Duncan, John},
   title={Complete normed algebras},
   note={Ergebnisse der Mathematik und ihrer Grenzgebiete, Band 80},
   publisher={Springer-Verlag, New York-Heidelberg},
   date={1973},
   pages={x+301},
   review={\MR{0423029}},
}
	

\bib{Bourbaki}{book}{
   author={Bourbaki, Nicolas},
   title={Lie groups and Lie algebras. Chapters 1--3},
   series={Elements of Mathematics},
   note={Translated from the French;
   Reprint of the 1975 edition},
   publisher={Springer-Verlag, Berlin},
   date={1989},
   pages={xviii+450},
   isbn={3-540-50218-1},
   review={\MR{979493}},
}

\bib{Segal}{book}{
   author={Carter, Roger},
   author={Segal, Graeme},
   author={Macdonald, Ian},
   title={Lectures on Lie groups and Lie algebras},
   series={London Mathematical Society Student Texts},
   volume={32},
   publisher={Cambridge University Press, Cambridge},
   date={1995},
   pages={vii+190},
   isbn={0-521-49579-2},
   isbn={0-521-49922-4},
   review={\MR{1356712}},
   doi={10.1017/CBO9781139172882},
}

\bib{Chernoff}{article}{
   author={Chernoff, Paul R.},
   title={Elements of a normed algebra whose $2^{n}{\rm th}$ powers lie
   close to the identity},
   journal={Proc. Amer. Math. Soc.},
   volume={23},
   date={1969},
   pages={386--387},
   issn={0002-9939},
   review={\MR{0246122}},
}

\bib{Cox}{article}{
   author={Cox, R.H.},
title={Matrices all of whose powers lie close the identity (Abstract)},
   journal={Amer. Math. Monthly},
   volume={73},
   date={1966},
   pages={813},
}

\bib{Laub}{article}{
   author={deLaubenfels, Ralph},
   title={Totally accretive operators},
   journal={Proc. Amer. Math. Soc.},
   volume={103},
   date={1988},
   number={2},
   pages={551--556},
   issn={0002-9939},
   review={\MR{943083}},
}

\bib{DPR}{article}{
   author={DePrima, C. R.},
   author={Richard, B. K.},
   title={A characterization of the positive cone of ${\bf B}({\germ h})$},
   journal={Indiana Univ. Math. J.},
   volume={23},
   date={1973/74},
   pages={163--172},
   issn={0022-2518},
   review={\MR{0315503}},
}

\bib{Dev}{article}{
   author={Devinck, Vincent},
   title={Universal Jamison spaces and Jamison sequences for
   $C_0$-semigroups},
   journal={Studia Math.},
   volume={214},
   date={2013},
   number={1},
   pages={77--99},
   issn={0039-3223},
   review={\MR{3043414}},
}

\bib{EiGr}{article}{
   author={Eisner, Tanja},
   author={Grivaux, Sophie},
   title={Hilbertian Jamison sequences and rigid dynamical systems},
   journal={J. Funct. Anal.},
   volume={261},
   date={2011},
   number={7},
   pages={2013--2052},
   issn={0022-1236},
   review={\MR{2822322}},
}


\bib{Gel}{article}{
   author={Gelfand, I.},
   title={Zur Theorie der Charaktere der Abelschen topologischen Gruppen},
   language={German, with Russian summary},
   journal={Rec. Math. [Mat. Sbornik] N. S.},
   volume={9 (51)},
   date={1941},
   pages={49--50},
   review={\MR{0004635}},
}

\bib{Gleason}{article}{
   author={Gleason, Andrew M.},
   title={Groups without small subgroups},
   journal={Ann. of Math. (2)},
   volume={56},
   date={1952},
   pages={193--212},
   issn={0003-486X},
   review={\MR{0049203}},
}

\bib{GWZ}{article}{
   author={Gomilko, Alexander},
   author={Wr\'obel, Iwona},
   author={Zem\'anek, Jaroslav},
   title={Numerical ranges in a strip},
   conference={
      title={Operator theory 20},
   },
   book={
      series={Theta Ser. Adv. Math.},
      volume={6},
      publisher={Theta, Bucharest},
   },
   date={2006},
   pages={111--121},
   review={\MR{2276936}},
}

\bib{Gor}{article}{
   author={Gorin, E. A.},
   title={Several remarks in connection with Gel\cprime fand's theorems on the
   group of invertible elements of a Banach algebra},
   language={Russian},
   journal={Funkcional. Anal. i Prilo\v zen.},
   volume={12},
   date={1978},
   number={1},
   pages={70--71},
   issn={0374-1990},
   review={\MR{0482227}},
}

\bib{GotoYamabe}{article}{
   author={Got\^o, Morikuni},
   author={Yamabe, Hidehiko},
   title={On some properties of locally compact groups with no small
   subgroup},
   journal={Nagoya Math. J.},
   volume={2},
   date={1951},
   pages={29--33},
   issn={0027-7630},
   review={\MR{0041864}},
}

\bib{HaHe}{article}{
   author={Hanoch, Gershon},
   author={Hershkowitz, Daniel},
   title={Forcing sequences of positive integers},
   journal={Czechoslovak Math. J.},
   volume={45(120)},
   date={1995},
   number={1},
   pages={149--169},
   issn={0011-4642},
   review={\MR{1314537}},
}

\bib{HeSch1}{article}{
   author={Hershkowitz, Daniel},
   author={Schneider, Hans},
   title={Matrices with a sequence of accretive powers},
   journal={Israel J. Math.},
   volume={55},
   date={1986},
   number={3},
   pages={327--344},
   issn={0021-2172},
   review={\MR{876399}},
}

\bib{HeSch2}{article}{
   author={Hershkowitz, Daniel},
   author={Schneider, Hans},
   title={Sequences, wedges and associated sets of complex numbers},
   journal={Czechoslovak Math. J.},
   volume={38(113)},
   date={1988},
   number={1},
   pages={138--156},
   issn={0011-4642},
   review={\MR{925947}},
}
	


\bib{Hirschfeld}{article}{
   author={Hirschfeld, R. A.},
   title={On semi-groups in Banach algebras close to the identity},
   journal={Proc. Japan Acad.},
   volume={44},
   date={1968},
   pages={755},
   issn={0021-4280},
   review={\MR{0239417}},
}
	
\bib{HofmannMorrisEMS}{book}{
   author={Hofmann, Karl H.},
   author={Morris, Sidney A.},
   title={The Lie theory of connected pro-Lie groups},
   series={EMS Tracts in Mathematics},
   volume={2},
   publisher={European Mathematical Society (EMS), Z\"urich},
   date={2007},
   pages={xvi+678},
   isbn={978-3-03719-032-6},
   review={\MR{2337107}},
}

\bib{Jam}{article}{
   author={Jamison, Benton},
   title={Eigenvalues of modulus $1$},
   journal={Proc. Amer. Math. Soc.},
   volume={16},
   date={1965},
   pages={375--377},
   issn={0002-9939},
   review={\MR{0176332}},
}

\bib{John}{article}{
   author={Johnson, Charles R.},
   title={Powers of matrices with positive definite real part},
   journal={Proc. Amer. Math. Soc.},
   volume={50},
   date={1975},
   pages={85--91},
   issn={0002-9939},
   review={\MR{0369395}},
}

\bib{KMOT}{article}{
   author={Kalton, N.},
   author={Montgomery-Smith, S.},
   author={Oleszkiewicz, K.},
   author={Tomilov, Y.},
   title={Power-bounded operators and related norm estimates},
   journal={J. London Math. Soc. (2)},
   volume={70},
   date={2004},
   number={2},
   pages={463--478},
   issn={0024-6107},
   review={\MR{2078905}},
}


\bib{Kaplansky}{book}{
   author={Kaplansky, Irving},
   title={Lie algebras and locally compact groups},
   series={Chicago Lectures in Mathematics},
   note={Reprint of the 1974 edition},
   publisher={University of Chicago Press, Chicago, IL},
   date={1995},
   pages={xii+148},
   isbn={0-226-42453-7},
   review={\MR{1324106}},
}

\bib{LiRoSpi}{article}{
   author={Li, Chi-Kwong},
   author={Rodman, Leiba},
   author={Spitkovsky, Ilya M.},
   title={On numerical ranges and roots},
   journal={J. Math. Anal. Appl.},
   volume={282},
   date={2003},
   number={1},
   pages={329--340},
   issn={0022-247X},
   review={\MR{2000347}},
}
	
\bib{LumVal}{article}{
   author={Luminet, Denis},
   author={Valette, Alain},
   title={Faithful uniformly continuous representations of Lie groups},
   journal={J. London Math. Soc. (2)},
   volume={49},
   date={1994},
   number={1},
   pages={100--108},
   issn={0024-6107},
   review={\MR{1253015}},
}
	
\bib{MaPa}{article}{
   author={Matsaev, V. I.},
   author={Palant, Ju. A.},
   title={On the powers of a bounded dissipative operator},
   language={Russian},
   journal={Ukrain. Mat. \v Z.},
   volume={14},
   date={1962},
   pages={329--337},
   issn={0041-6053},
   review={\MR{0146664}},
}

\bib{MT}{article}{
   author={Michor, Peter},
   author={Teichmann, Josef},
   title={Description of infinite-dimensional abelian regular Lie groups},
   journal={J. Lie Theory},
   volume={9},
   date={1999},
   number={2},
   pages={487--489},
   issn={0949-5932},
   review={\MR{1718235}},
}

\bib{MontgomeryZippin}{book}{
   author={Montgomery, Deane},
   author={Zippin, Leo},
   title={Topological transformation groups},
   note={Reprint of the 1955 original},
   publisher={Robert E. Krieger Publishing Co., Huntington, N.Y.},
   date={1974},
   pages={xi+289},
   review={\MR{0379739}},
}

\bib{MorrisPestov}{article}{
   author={Morris, Sidney A.},
   author={Pestov, Vladimir},
   title={On Lie groups in varieties of topological groups},
   journal={Colloq. Math.},
   volume={78},
   date={1998},
   number={1},
   pages={39--47},
   issn={0010-1354},
   review={\MR{1658131}},
}

\bib{NagWa}{article}{
   author={Nagisa, Masaru},
   author={Wada, Shuhei},
   title={Averages of operators and their positivity},
   journal={Proc. Amer. Math. Soc.},
   volume={126},
   date={1998},
   number={2},
   pages={499--506},
   issn={0002-9939},
   review={\MR{1415335}},
}

\bib{NakamuraYoshida}{article}{
   author={Nakamura, Masahiro},
   author={Yoshida, Midori},
   title={On a generalization of a theorem of Cox},
   journal={Proc. Japan Acad.},
   volume={43},
   date={1967},
   pages={108--110},
   issn={0021-4280},
   review={\MR{0217619}},
}


\bib{Neeb}{article}{
   author={Neeb, Karl-Hermann},
   title={Lectures on Infinite Dimensional Lie Groups},
   journal={lecture notes, Monastir Summer School, available at https://cel.archives-ouvertes.fr/cel-00391789/document},
   date={2005},   
}

\bib{Paulos}{article}{
   author={Paulos, J.},
   title={Stability of Jamison sequences under certain perturbations},
   journal={North-W. Eur. J. of Math.},
   volume={5},
   date={2019},
   pages={89--99},
}


\bib{Ran}{article}{
   author={Ransford, Thomas},
   title={Eigenvalues and power growth},
   journal={Israel J. Math.},
   volume={146},
   date={2005},
   pages={93--110},
   issn={0021-2172},
   review={\MR{2151595}},
}

\bib{RR}{article}{
   author={Ransford, Thomas},
   author={Roginskaya, Maria},
   title={Point spectra of partially power-bounded operators},
   journal={J. Funct. Anal.},
   volume={230},
   date={2006},
   number={2},
   pages={432--445},
   issn={0022-1236},
   review={\MR{2186219}},
}

\bib{Rosendal}{article}{
   author={Rosendal, Christian},
   title={Lipschitz structure and minimal metrics on topological groups},
   journal={Ark. Mat.},
   volume={56},
   date={2018},
   number={1},
   pages={185--206},
   issn={0004-2080},
   review={\MR{3800465}},
}

\bib{Shiu}{article}{
   author={Shiu, Elias S. W.},
   title={Growth of numerical ranges of powers of Hilbert space operators},
   journal={Michigan Math. J.},
   volume={23},
   date={1976},
   number={2},
   pages={155--160},
   issn={0026-2285},
   review={\MR{0412845}},
}

\bib{Tao5}{book}{
   author={Tao, Terence},
   title={Hilbert's fifth problem and related topics},
   series={Graduate Studies in Mathematics},
   volume={153},
   publisher={American Mathematical Society, Providence, RI},
   date={2014},
   pages={xiv+338},
   isbn={978-1-4704-1564-8},
   review={\MR{3237440}},
}
	
\bib{Uch}{article}{
   author={Uchiyama, Mitsuru},
   title={Powers and commutativity of selfadjoint operators},
   journal={Math. Ann.},
   volume={300},
   date={1994},
   number={4},
   pages={643--647},
   issn={0025-5831},
   review={\MR{1314739}},
}

\bib{Wallen}{article}{
   author={Wallen, Lawrence J.},
   title={On the magnitude of $x^{n}-1$ in a normed algebra},
   journal={Proc. Amer. Math. Soc.},
   volume={18},
   date={1967},
   pages={956},
   issn={0002-9939},
   review={\MR{0216295}},
}

\bib{Wils}{article}{
   author={Wils, W.},
   title={On semigroups near the identity},
   journal={Proc. Amer. Math. Soc.},
   volume={21},
   date={1969},
   pages={762--763},
   issn={0002-9939},
   review={\MR{0239468}},
}
	
\bib{Zem}{article}{
   author={Zem\'anek, Jaroslav},
   title={On the Gelfand-Hille theorems},
   conference={
      title={Functional analysis and operator theory},
      address={Warsaw},
      date={1992},
   },
   book={
      series={Banach Center Publ.},
      volume={30},
      publisher={Polish Acad. Sci. Inst. Math., Warsaw},
   },
   date={1994},
   pages={369--385},
   review={\MR{1285622}},
}
  \end{biblist}
\end{bibdiv}
\end{document}